\documentclass[12pt]{amsart}

\usepackage[pdftex]{graphicx}

\usepackage[T2A]{fontenc}
\usepackage[utf8]{inputenc}
\usepackage[english]{babel}

\usepackage{amsmath}
\usepackage{amssymb}
\usepackage{amscd}
\usepackage{amsfonts,amsmath,amsthm}

\usepackage[matrix,arrow,curve]{xy}
\usepackage{xfrac}
\usepackage{color}

\newtheorem{Ex}{Example}
\newtheorem{theorem}{Theorem}
\newtheorem{prop}{Proposition}

\newtheorem{lemma}{Lemma}
\newtheorem*{main lemma}{The main lemma}

\newtheorem{definition}{Definition}

\title{Counting normals to closed curves in $\mathbb{R}^3$ }

\author{ Gaiane Panina, Dirk Siersma}


\thanks{The second author acknowledges support from the project ``Singularities and Applications'' - CF 132/31.07.2023 funded by the European Union - NextGenerationEU - through Romania's National Recovery and Resilience Plan, and support by the grant CNRS-INSMI-IEA-329. 
\newline  The research was carried on in Batumi during the Winter Math Weeks 2025 and 2026. We thank our host George Khimshiashvili for hospitality and useful remarks.}
\date{\today}

\begin{document}

\begin{abstract}

  We prove the following results:
  
  (1) For every generic closed smooth curve in $\mathbb{R}^3$, {whose torsion has a constant sign}, there is a point with at least $6$  emanating normals to the curve.
  
If the curve is knotted, and torsion has a constant sign, there is a point with at least $8$  emanating normals.
 
   (2) For every generic closed piecewise linear curve in $\mathbb{R}^3$ there is a point with at least $8$  emanating normals to the curve {(no condition on the torsion)}. 
If the curve is knotted, then there is a point with at least $10$  emanating normals.

 The proof is based on the Morse theory for the squared distance function and self intersections of the focal surface.
\end{abstract}

\maketitle

\section{Preliminaries and basic definitions}

It  has long been conjectured that for every convex body $P\subset \mathbb{R}^n$, there exists a point in its interior   which belongs to at least $2n$ normals from different points on the boundary of $P$. The conjecture is proven for $n\leq 4$ \cite{Heil,Pardon}, and almost nothing is known for higher dimensions.

A similar problem, with appropriate definitions, makes sense for convex polytopes.
 It was proven recently \cite{NasPan,NasPan1} that each convex simple polytope $P\subset \mathbb{R}^n$  has a point with at least $2n+4$ normals (provided that $n>2$).

In particular, for curves in the plane, elementary facts are:

 For every smooth convex curve in the plane,  there is a point  with at least four normals to the curve. The bound is exact since an ellipse has no point with more than four normals. 

For every piecewise linear convex curve in the plane,  there is a point  with at least six normals to the curve. This bound is also exact since a triangle has no point with more than six normals. 

In the paper we turn to closed curves in $\mathbb{R}^3$, either smooth or piecewise linear.

\begin{definition}

A segment $yx$ is called a normal to a curve $C$, emanating from the point  $y$,  if $x\in C$ and $x$ is either a local maximum or a local minimum of the squared distance function
$$SQD_y:C\rightarrow \mathbb{R}, \ \ \  SQD_y(x):=|x-y|^2.$$

The number of normals emanating from $y$ is denoted by $N(y)$.
\end{definition}

For a smooth curve $C$, a normal is orthogonal to the tangent vector at the point $x$. 

 In the piecewise linear case the definition agrees with the  geometric intuition: a segment $yx$ is a  normal to $C$ if and only if  the plane containing $x$ and orthogonal to  $xy$ intersects the curve  $C$ at  the point $x$ non-transversally.

\medskip

Our main results are:

 \begin{theorem} \label{Thm1}

   For every generic smooth closed curve in $\mathbb{R}^3$ {whose torsion has a constant sign}, there exists  a point $y$ with at least $6$ emanating normals.
      {  If the curve is knotted,    there is  a point $y$ with at least $8$ emanating normals.}
\end{theorem}

As our experience  \cite{NasPanSiersma,NasPan,NasPan1} suggests, piecewise linear objects  have more normals than smooth ones.
So  the reader should not be surprised to see the following theorem:

\begin{theorem}\label{Thm1Knot}
  For every generic non-planar piecewise linear curve in $\mathbb{R}^3$, there is a point with at least $8$ emanating normals.

   If the curve is knotted, there is a point with at least $10$ normals.
\end{theorem}

\medskip

There arises  natural questions:
{ Does Theoren \ref{Thm1} hold for all generic smooth curves?  Is the bound in Theorem \ref{Thm1Knot} exact for PL curves with more than $4$ edges?}

\section{Smooth curves. Proof of Theorems \ref{Thm1} }

Throughout the paper we assume that $C$ is a generic  curve embedded in $\mathbb{R}^3$. From now on, we call local maxima and local minima of
$SQD_y$  maxima and minima,  for brevity. 

A simple observation is:  the number of  minima is equal to the number of  maxima, and they alternate on the curve (between each two minima there is a maximum, and vice versa).

\begin{lemma}\label{LemmaKnot4} 
  For an embedded closed curve $C$,  either smooth or piecewise linear, if there exists $y$ with $N(y)=2$, then $C$ is an unknot.
\end{lemma}
\begin{proof}
We mimic one of the proofs of the Milnor-F\'{a}ry theorem, see \cite{Petrunin}.  Consider the family of concentric spheres $S_r$ of radius $r$ centered at $y$. By assumption, each such sphere intersects $C$ at most twice.
By a slight perturbation of $C$, we may assume that the two intersection points are never antipodal. Connect the intersection points by the (uniquely defined) shortest geodesic segment.
The union of all such segments is an embedded disk $d$ with the property $C=\partial d$, hence $C$ is unknotted.
\end{proof}
\bigskip

For a smooth curve $C$ and a point $x\in C$, let us fix the following notation. Let $h(x)$ denote the osculating plane, and let $x^{\perp}$ denote the normal plane. The point $F(x)\in h(x)$ is the center of the osculating circle. A generic curve has everywhere nonzero curvature, so we may assume that $F(x)$ is finite.

Let $L(x)$ be the line passing through $F(x)$ and orthogonal to $h(x)$. Thus, $L(x)$ is directed along the binormal.

The \textit{focal surface} $\mathcal{F}(C)$ is the union of all these lines.

The surface $\mathcal{F}(C)$ is the \textit{bifurcation set} of the function $SQD_y$. That is, $\mathcal{F}(C)$ divides the space into \textit{chambers}. In each chamber, the number of normals $N(y)$ remains constant. If $y$ crosses a sheet of $\mathcal{F}(C)$ transversally, then either two normals are born or two normals disappear.

The \textit{evolute} $\mathcal{E}(C)$ is the set of centers of osculating spheres. It is a curve contained in $\mathcal{F}(C)$.We refer the reader to \cite{IzmFuchs,Uribe-Vargas} for more details on spatial curves.
 \medskip

\begin{prop}
{ The focal surface $\mathcal{F}(C)$ of any generic closed curve in $\mathbb{R}^3$ has a transverse double self-intersection.}
\end{prop}
\begin{proof}
 The focal surface $\mathcal{F}(C)$ is a developable surface since it is the envelope of normal planes to the curve $C$  (in the literature it is called normal developable, see \cite{IzmFuchs}). At the same time it is the \textit{tangent developable},
that is, the union of all tangent lines to the evolute $\mathcal{E}(C)$, which is
generically a closed curve, possibly with cusps. { The evolute $\mathcal{E}(C)$ has no points at infinity, since the latter correspond to the points on $C$ with zero torsion}.

\newpage

 Let us analyze the following two cases:

(1)   Assume that $\mathcal{E}(C)$ has no cusps.  We have to prove 
that $\mathcal{E}(C)$ has two intersecting tangent lines.

 Consider the convex hull $Conv(\mathcal{E}(C))$, and let $x$ be its extremal point. Take
passing through $x$. By definition of an extremal point, there exists  a support plane $S$ to $Conv(\mathcal{E}(C))$ which is  tangential to $\mathcal{E}(C)$  at  the point $x$ and intersects $\mathcal{E}(C)$ at no other point.  We may assume that the torsion of $\mathcal{E}(C)$   does not vanish at $x$. 
Consider the set $T$  of planes that contain $x$ and are tangent to the curve $\mathcal{E}(C)$ at $x$.

 One of the planes from $T$ is the osculating plane $OP$ at the point $x$, and since the torsion is non-zero, $OP\neq S$, and moreover, none of the planes from $T$ sufficiently close to $OP$ is a support plane to  $Conv(\mathcal{E}(C))$.

Consider the subset   $\mathcal{SP} \subset T$  consisting of planes that are support planes to  $Conv(\mathcal{E}(C))$. This set is homeomorphic to a closed segment which might degenerate to  a single point $S$.

Let us prove that any endpoint $H$ of $\mathcal{SP}$ is tangential to $\mathcal{E}(C)$ at some other point, distinct from  $x$. 

Rotate $H$ slightly further, while keeping the rotated plane in $T$, such that the rotated plane is no longer a support plane and therefore intersects $\mathcal{E}(C)$ elsewhere. We may assume that the rotation is by a small angle, so the rotated planes do not intersect a neighborhood of $x$ on the curve $\mathcal{E}(C)$ at any other points than $x$.

 When the plane rotates back to its endpoint position $H$, the intersection points collide and disappear. None of these collisions occurs at $x$. Thus, each such collision gives rise to a point of tangency distinct from $x$.

(2)   Now assume  that  $\mathcal{E}(C)$ has a cusp.  A cusp of the evolute corresponds to the swallow tail of  $\mathcal{F}(C)$, which provides transverse self-intersection.
   
\end{proof}   
   Now we are ready to prove the theorem.

   A neighborhood of a point of transverse self-intersection  of $\mathcal{F}(C)$
  is divided into four chambers by two sheets. Crossing of each  sheet changes $N(y)$ by $2$, so in these chambers we have $N$, $N+2$ and $N+4$ emanating normals for some $N$.  The value of  $N$ is at least $2$, whereas in the knotted case  it is at least $4$   by Lemma \ref{LemmaKnot4} which implies the desired.
   
   \medskip

{\textbf{Remark.}  It is clear from the above proof that  the statement of Theorem \ref{Thm1}  holds true for any generic closed smooth  curve $C$ provided that the focal surface  $\mathcal{F}(C)$ has self-intersections.}

\begin{Ex}
  A planar ellipse in $\mathbb{R}^3$ has no point with  $6$ normals. However, its small perturbation yields points with $6$ normals. 
\end{Ex}

A natural question is about the maximum of $N(y)$. Clearly, inside the isotopy class of the curve one can arrange arbitrary high numbers. If one restricts to algebraic curves there will be  some upper bounds. This bound is especially studied in the context of Euclidean Distance degree (ED-degree) \cite{DHOST}. One considers as upper bound the number of critical points of the complexification of the variety with respect to the complex squared distance function. This number does not depend on $y$.

For smooth complete intersection curves in $\mathbb{R}^3$ with degree $(d_1,d_2)$ the ED-degree is $d_1 d_2 (d_1+d_2-2)$.

In the case of the lifted ellipse one could take $d_1=d_2=2$, so  the bound is $8$.

\section{{Piecewise linear curves. Proof of Theorem \ref{Thm1Knot}  }}

In this section we assume that $C$ is a piecewise linear non-planar generic curve embedded in $\mathbb{R}^3$.
In this case the maxima of $SQD_y$ are always attained at the vertices, whereas minima may be attained at both vertices and edges.

We start with the proof of the first statement of Theorem \ref{Thm1Knot}:
\begin{proof}
 There exists a superscribed sphere, that is, a sphere containing  $4$ vertices of the curve.  Take its center  $O$. The function $SQD_O$ has at least four maxima, hence four
 minima.
\end{proof}

Now let us  consider a knotted curve $C$ and prove the second statement.

\begin{proof}
  Let  us first describe the bifurcation set $\mathcal{B}(C)$. Each vertex $v$ of the curve with two emanating edges $l$ and $l'$ yields two half-planes 
  with a common boundary.  This is how it arises:  take normal planes $h$ and $h'$ to the edges $l$ and $l'$ such that $v\in l, l'$. Taken together, they divide the space into four domains. In exactly one of these domains the vertex $v$ is a maximum of $SQD$. The union of the two half-planes  is the boundary of this domain;  we call it \textit{the wedge  related to the vertex $v$}. The intersection $h\cap h'$ is called \textit{the ridge of the wedge}, see Fig. \ref{ridge}.
  
  \begin{figure}[h]
 \includegraphics[width=0.6\linewidth]{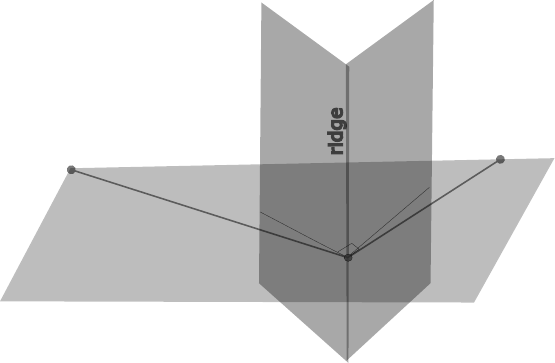}
  \caption {Each vertex of $C$ gives a wedge of  $\mathcal{B}(C)$. }
  \label{ridge}
\end{figure}

Each wedge cuts the space $\mathbb{R}^3$ into two connected components. If $y$ belongs to the smaller one, $yv$  is a normal which corresponds to a maximum.
  
   Whenever a point $y$ crosses a wedge transversely,
  there is a bifurcation involving the vertex $v$.  This means that two normals (one maximum related to $v$  and the other minimum related to either $l$ or $l'$) either die of is born.

  By definition, the union of all the wedges equals $\mathcal{B}(C)$.  The set $\mathcal{B}(C)$ cuts the space $\mathbb{R}^3$ into\textit{ chambers}.  
   Each chamber is some polytope, possibly non-convex, possibly unbounded.
   If a point moves in one and the same chamber (that is, crosses no wedge), the number of normals persists.

  \begin{prop}
    There is a point of triple intersection of three different wedges of $\mathcal{B}(C)$. 
  \end{prop}
  
  Let us first show that\textbf{ the proposition implies the theorem.}  A point of triple intersection has $8$ incident chambers. The number of normals for these chambers is $N, \ N+2,\ N+4,$ and $N+6$ for some $N$. By Lemma \ref{LemmaKnot4},  we have  $N\geq 4$, which completes the proof.
  \medskip

  Now let us prove the proposition. 
  
   Given a chamber,
  some of its edges   are intersections of two different wedges, whereas some edges come from ridges. Call the latter ones \textit{ridge edges}.
  
  Genericity implies:
  \begin{lemma}
  \begin{enumerate}
    \item No chamber contains a straight line.
    \item A face of a chamber has no more than one  ridge edge. 
    \item Each vertex of a chamber has three emanating edges of the chamber.  At most one  of them is a ridge edge.
  \end{enumerate}
  \end{lemma}
  \begin{proof}
   The statements (2) and (3) follow from genericity, so let us prove (1).
    
    It suffices to prove that each line $L$ intersects the bifurcation set. Take the vertex $v$ of $C$ such that the function $SQD_y$ attains a local maximum at $v$ for some point  $y\in L$. If $L$ does not cross $\mathcal{B}(C)$, the function $SQD_y$ attains maximum at $v$ for all $y\in L$.  This happens only if $L$ is parallel to a ridge associated to the vertex $v$. On the other hand, in the knotted case such a vertex $v$ is not unique  by Lemma \ref{LemmaKnot4}. By genericity, no two ridges are parallel, so we have a contradiction.
  \end{proof}

    \begin{lemma}\label{LemmaR}
    If a bounded face of a chamber is not a triangle, then it has a vertex with no incident ridge edge. Such a vertex is a transverse intersection of three wedges.
   \end{lemma}
   \begin{proof}
     Let the face be a quadrilateral $ABCD$  (the bigger numbers of vertices are treated similarly).
     
      Assuming the contrary (each vertex has an emanating ridge edge), consider two cases:
      
      (1)  None of the edges of $ABCD$ is a ridge.  These points have one more incident edge of the chamber. All of them are ridges, which contradicts Lemma \ref{LemmaR}, (2).
      
      (2) Assume $AB$ is a ridge. Then either $C$ or $D$ has no incident ridge edges by Lemma \ref{LemmaR}, (2).

   \end{proof}
   So it remains to show that not all the bounded faces of $\mathcal{B}(C)$ are triangles.

    Consider the sphere at infinity. Each wedge intersects it by \textit{a lune} which is the union of two geodesics connecting the pair of antipodal points that correspond to the ridge of the wedge. There is a natural circle ordering on the lunes since they come from vertices on the curve. Two neighbour lunes necessarily have a common geodesic segment, see Fig. \ref{lunes}. If the lunes are not neighbors, we may assume by genericity that they intersect transversally.

      \begin{figure}[h]
 \includegraphics[width=0.9\linewidth]{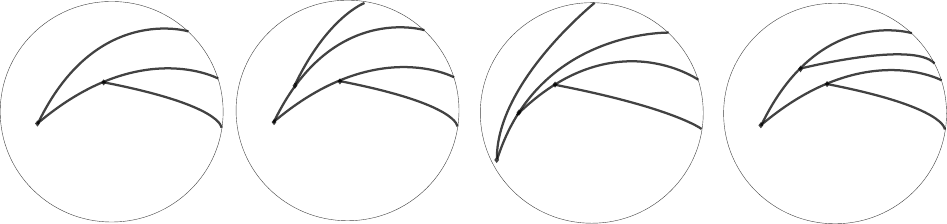}
  \caption {Lunes on the  sphere at infinity: two lunes, and some different positions of three neighbour lunes. }
  \label{lunes}
\end{figure}
    
    The union of lunes yields a tiling  $T$ of the sphere into some spherical polygons. A simple analysis shows that the vertices of the tiling have valence $4$ with the following possible exception at
    the endpoints of the lunes:

    \begin{lemma}
    Each lune has either has two endpoints of valency $3$, or one endpoint of valency $2$ and the other endpoint of valency $4$.
    All the other vertices of $T$ have valency $4$. 
     \end{lemma}
     
     \begin{proof}
       See Fig. \ref{lunes} for some possible relative positions of lunes. The reader can easily fill in all the missing cases.
     \end{proof}
    
      \begin{lemma}
    Each tile of $T$ fits in an open hemisphere.
    Therefore, no tile is a bigon (no tile is a lune).
     \end{lemma}
     \begin{proof}  {
      By Lemma \ref{LemmaKnot4}, each point of the sphere (and therefore, each tile)  is covered by at least two lunes. }
     \end{proof}

     \begin{lemma}
    Not all the tiles of $T$ are triangles and quadrilaterals.
     \end{lemma}
   \begin{proof}
     We use the following  common approach: assuming the contrary, replace each triangular tile by an equilateral triangle  with unit edge lengths, and each quadrilateral by a square  also with the unit edge lengths.  They patch together to a sphere with almost everywhere flat metric. The curvature is concentrated at the vertices, and each vertex contributes the \textit{angular defect} (it equals the difference of $2\pi$  and the sum of the incident angles). The total curvature  can be estimated as follows: each vertex of valency $4$ contributes at least zero; each vertex of valency $3$ contributes at least $\pi/2$, and each vertex of valency $2$ contributes at least $\pi$. Since the sum of the curvatures over all the vertices equals $4\pi$,  we conclude that there are at most $4$ lunes. Since the number of lunes equals the number of vertices of $C$, we have a contradiction.
   \end{proof}

   We proceed with the proof of the proposition. So the tiling $T$ has a tile $\tau$ with at least $5$ vertices.  Let us assume that there are exactly $5$ vertices; the bigger numbers are treated similarly. There is an unbounded chamber $\sigma$ adjacent to this tile.
   All the vertices of $\sigma$ are trivalent because of genericity.
   
    \begin{figure}[h]
 \includegraphics[width=0.5\linewidth]{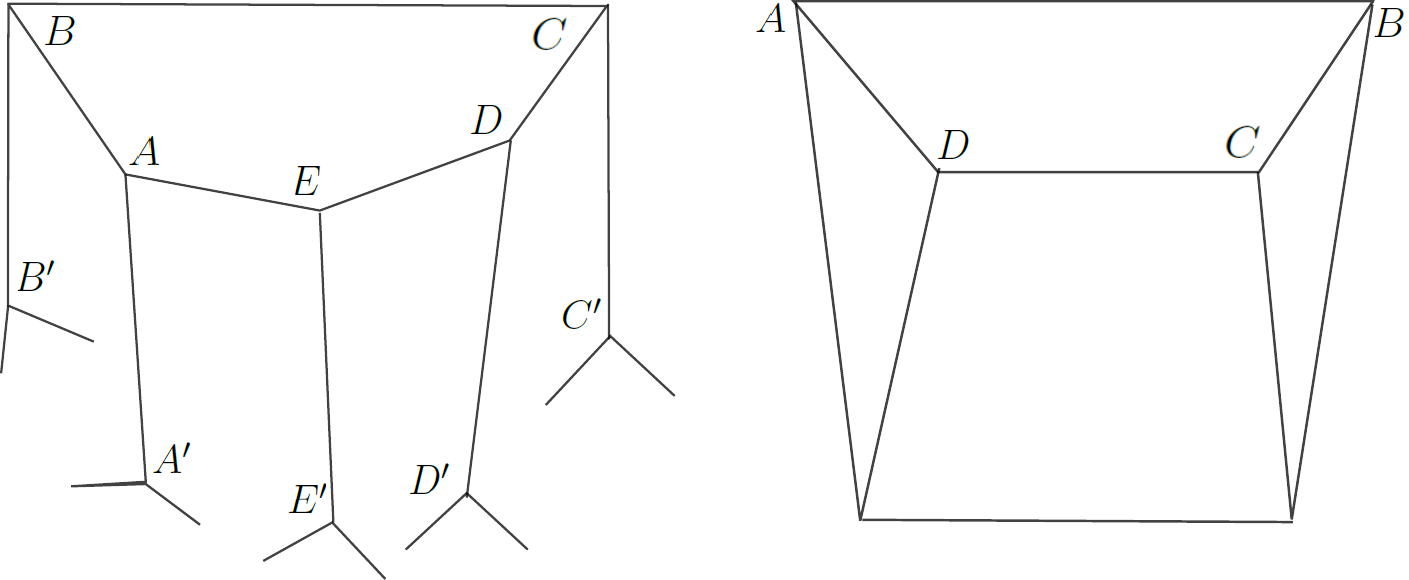}
  \caption {The chamber $\sigma$  (left). The proof does not work for quadrilateral $\tau$ (right).  }
  \label{chamber}
\end{figure}
   By similar curvature count we conclude that $\sigma$ has a bounded face with at least $4$ vertices.
   
   First observe that all the points $A',B',C',D',E'$ are distinct (see Fig. \ref{chamber}).
   
   Next, replace the tile $\tau$ by a regular pentagon $ABCDE$, replace  all other faces by regular polygons with equal edge lengths.
   Assume that all the bounded faces are triangles. Then the contribution of $A,B,C,D,E$ equals at least $$5\cdot (2\pi-(\pi/2+\pi/2+{3}\pi/{5} ))=2\pi.$$
    The contribution of $A',B',C',D',E'$ is at least $$5\cdot(2\pi-(\pi+\pi/3))=\frac{10\pi}{3},$$
    which is already bigger than $4\pi$.\footnote{The genus of the chamber might be non-zero, but in this case we come to a contradiction as well.}
    
    \medskip
  \textbf{Remark.}  For  a quadrilateral $\tau$  the proof breaks down: the points
  $A',B',C',D'$ might be not all distinct, see Fig. \ref{chamber}, right.
  
  The lemma (and hence the proposition)   are proven.
\end{proof}

As a by-product of the above, we get the following.
\begin{prop}
  If $C$ is  a knotted piecewise linear curve, then there exists a linear function whose restriction on $C$ has at least $8$ critical points (that is, minima and maxima).
\end{prop}

\begin{proof}
There are points of transverse  intersections of two different sheets on the infinity. Since for every chamber, the number of normals is at least $4$, and crossing of a sheet amounts to death (or birth) of yet another maximum and minimum, we get the desired.
\end{proof}

\textbf{Remark.}
For closed piecewise linear curves the number of normals is always finite and one has $N(y) \le 2v$, where $v$ is the number of vertices.

For any $n$ there are polygons, such that $N(y) = 2v$. The concept of aED-degree is defined by the average of $N(y)$ with respect to a measure, cf \cite{DHOST}. So aED-degree $\le 2v$.



\begin{thebibliography}{99}

\bibitem{DHOST}
J.~Draisma, E.~Horobet, G.~Ottaviani, B.~Sturmfels, and R.~R.~Thomas,
\emph{The Euclidean distance degree of an algebraic variety},
Found. Comput. Math. \textbf{16} (2016), no.~1, 99--149.


\bibitem{IzmFuchs}
D.~Fuchs, I.~Izmestiev, and M.~Raffaelli,
\emph{Differential geometry of space curves: forgotten chapters},
Math. Intelligencer \textbf{46} (2024), 9--21.

\bibitem{Heil}
E.~Heil,
\emph{Concurrent normals and critical points under weak smoothness assumptions},
In: \emph{Discrete Geometry and Convexity} (New York, 1982),
Ann. New York Acad. Sci. \textbf{440} (1985), 170--178.


\bibitem{NasPanSiersma}
I.~Nasonov, G.~Panina, and D.~Siersma,
\emph{Concurrent normals problem for convex polytopes and Euclidean distance degree},
Acta Math. Hungar. \textbf{174} (2024), 522--538.

\bibitem{NasPan}
I.~Nasonov and G.~Panina,
\emph{Each simple polytope from $\mathbb{R}^n$ has a point with $2n+4$ normals to the boundary},
Zap. Nauchn. Sem. POMI \textbf{549} (2025), 161--169.

\bibitem{NasPan1}
I.~Nasonov and G.~Panina,
\emph{Each generic polytope in $\mathbb{R}^3$ has a point with ten normals to the boundary},
arXiv:2411.12745.

\bibitem{Pardon}
J.~Pardon,
\emph{Concurrent normals to convex bodies and spaces of Morse functions},
Math. Ann. \textbf{352} (2012), no.~1, 55--71.

\bibitem{Petrunin}
A.~Petrunin and S.~Stadler,
\emph{Six proofs of the F'{a}ry--Milnor theorem},
Amer. Math. Monthly \textbf{131} (2024), no.~3, 239--251.

\bibitem{Uribe-Vargas}
R.~Uribe-Vargas,
\emph{On vertices, focal curvatures and differential geometry of space curves},
Bull. Braz. Math. Soc. \textbf{36} (2005), 285--307.

\end{thebibliography}
\end{document}